\documentclass[11pt,reqno]{amsart}

\usepackage[T1]{fontenc}
\usepackage{lmodern}
\usepackage{microtype}
\usepackage{amsmath,amssymb,amsthm,mathtools,mathrsfs}
\usepackage[a4paper,margin=28mm,headheight=14pt]{geometry}
\usepackage{enumitem}
\usepackage{xcolor}
\usepackage[
  colorlinks=true,
  linkcolor=red,
  citecolor=blue,
  urlcolor=blue
]{hyperref}

\allowdisplaybreaks[2]
\numberwithin{equation}{section}
\setlist[enumerate]{label=\textup{(\roman*)},leftmargin=*,itemsep=3pt}

\newtheorem{theorem}{Theorem}[section]
\newtheorem{proposition}[theorem]{Proposition}
\newtheorem{lemma}[theorem]{Lemma}

\newtheorem*{conj}{Conjecture}
\theoremstyle{definition}
\newtheorem{definition}[theorem]{Definition}
\theoremstyle{remark}
\newtheorem{remark}[theorem]{Remark}
\newcommand{\OO}{\mathcal O}

\newcommand{\HH}{\mathcal H}
\newcommand{\CC}{\mathbb C}
\newcommand{\End}{\operatorname{End}}
\newcommand{\Hom}{\operatorname{Hom}}
\newcommand{\tr}{\operatorname{tr}}

\newcommand{\rk}{\operatorname{rk}}
\newcommand{\Mdr}{M_{\mathrm{dR}}}
\newcommand{\Mhod}{M_{\mathrm{Hod}}}
\newcommand{\ord}{\operatorname{ord}}
\newcommand{\Gr}{\operatorname{Gr}}

\title{Simpson's closedness conjecture in arbitrary rank}
\author[Tianzhi Hu]{Tianzhi Hu}
\address{
School of Mathematics and Statistics, Wuhan University, Wuhan, Hubei 430072, P.R. China
}
\email{hutianzhi@whu.edu.cn}
\date{}

\begin{document}

\begin{abstract}
For a compact Riemann surface $X$, Simpson associated to each stable graded Higgs bundle $(E,\theta)$ a locus $W_{[(E,\theta)]}^1 \subset M_{\mathrm{dR}}(X,n)$,
consisting of flat bundles that admit a Simpson filtration whose associated graded Higgs bundle is $(E,\theta)$. He conjectured that $W_{[(E,\theta)]}^1$ is Zariski closed in $M_{\mathrm{dR}}(X,n)$. We prove this conjecture in arbitrary rank.

The proof is based on the Quillen geometry of the determinant-of-cohomology
line bundle over the moduli of holomorphic bundles. To any holomorphic
family of flat bundles, we associate a determinant-of-cohomology line
bundle equipped with a canonical holomorphic determinant connection, and
derive explicit formulas for its curvature. For a family of flat bundles in
$W_{[(E,\theta)]}^1$, we then prove that the determinant connection form is
exact. This exactness forces the associated determinant frame to extend as a nowhere-vanishing frame across any one-parameter degeneration, which controls the limiting filtration and
yields the desired closedness.
\end{abstract}

\maketitle

\section{Introduction}\label{sec:assertion}

Let \(X\) be a compact Riemann surface of genus
\(g\geq2\).
Fix an integer \(n\geq1\). The non-abelian Hodge correspondence relates
reductive representations of \(\pi_1(X)\) to polystable degree-zero
Higgs bundles through harmonic metrics
\cite{UhlenbeckYau1986,Hitchin1987,Donaldson1987,Corlette1988,Simpson1992}.
It identifies the Dolbeault moduli space \(M_{\mathrm{Dol}}(X,n)\) of
Higgs bundles with the de Rham moduli space \(\Mdr(X,n)\) of flat
bundles as topological spaces, although their algebraic structures differ.

Simpson's Hodge moduli space \(\Mhod=\Mhod(X,n)\) places these two
spaces in an algebraic family \cite{Simpson1997,Simpson1994I}.
It parametrizes semistable degree-zero rank-\(n\) \(\lambda\)-connections
\(D_\lambda:V\to V\otimes K_X\), satisfying
\[
 D_\lambda(fs)=\lambda\,df\otimes s+fD_\lambda s.
\]
Its fibers at \(0\) and \(1\) are \(M_{\mathrm{Dol}}(X,n)\) and
\(\Mdr(X,n)\), respectively. Simpson introduces the following $\CC^*$-action on \(M_{\mathrm{Dol}}(X,n)\):
\[
 z\cdot(\lambda,V,\nabla_\lambda)
       =(z\lambda,V,z\nabla_\lambda),\qquad z\in\CC^*.
\]
On the Dolbeault fiber, this action scales the Higgs field.

The fixed points in \(M_{\mathrm{Dol}}(X,n)\) are represented by graded
Higgs bundles, or systems of Hodge bundles \cite{Simpson1992}.
Fix such a bundle \((E,\theta)\), assumed stable of degree zero, and write
\begin{equation}\label{eq:hodge}
 E=\bigoplus_{i=1}^{s}E_i,\qquad
 \theta_i:E_i\longrightarrow E_{i+1}\otimes K_X,\qquad
 \theta|_{E_s}=0,
\end{equation}
where \(1\leq s\leq n\) and
\(\sum_i\rk E_i=n\). The only nonzero blocks allowed for \(\theta\)
are the displayed adjacent maps. 

Following Simpson \cite{Simpson2010}, define the attracting fiber
\[
 W_{[(E,\theta)]}^1=
 \{[(V,D)]\in\Mdr(X,n):
       \lim_{z\to0}[(z,V,zD)]=[(E,\theta)]\},
\]
with the limit taken in \(\Mhod\). Equivalently, its points admit a
Simpson filtration: a Griffiths-transverse filtration whose associated
graded Higgs bundle is \((E,\theta)\)
\cite[Theorem~2.5 and Proposition~4.3]{Simpson2010}.
Simpson's foliation question
\cite[Question~7.4]{Simpson2010} asks whether the following closedness property holds:

\begin{conj}[Simpson]
For each stable degree-zero graded Higgs bundle \((E,\theta)\) of
rank \(n\), the attracting fiber \(W_{[(E,\theta)]}^1\) is Zariski closed
in \(\Mdr(X,n)\).
\end{conj}
If true, this closedness would imply a strong separatedness property for the Simpson filtration data and provide a canonical algebraic decomposition of the de Rham moduli space by fixed graded Higgs limits.

Simpson established local closedness and the Lagrangian property
\cite[Proposition~5.1 and Lemma~7.3]{Simpson2010}.
Collier--Wentworth obtained affine parametrizations through conformal
limits \cite{CollierWentworth2019}. Hu--Huang proved that the oper
stratum is the unique stratum of minimal dimension \cite{HuHuang2022};
this concerns strata rather than individual attracting fibers.
Schulz developed a WKB approach \cite{Schulz2023}, and Dimakis--Schulz
proved the closedness conjecture for every stable fixed point in rank two
\cite{DimakisSchulz2024}. 

We prove the following theorem, which gives an affirmative answer to Simpson's closedness conjecture.

\begin{theorem}\label{thm:main}
For every \(n\geq1\) and every stable $\mathrm{GL}(n,\CC)$ graded Higgs bundle
\((E,\theta)\) as in \eqref{eq:hodge}, the subset
\(W_{[(E,\theta)]}^1\) is Zariski closed in \(\Mdr(X,n)\).
\end{theorem}

Section~\ref{sec:determinant-connection} uses Quillen geometry
\cite{Quillen1985,BismutGilletSoule1988III} to study the holomorphic
determinant connection \(\nabla^{\mathrm{det}}\) of the determinant-of-cohomology
line bundle for families of flat bundles.
Section~\ref{sec:rees-triangular-exactness} proves that its connection form is
exact for families admitting the same Simpson filtration. We further establish
Proposition~\ref{prop:valuation}, which is used to control one-parameter
degenerations. Section~\ref{sec:determinant-degree} proves that the limiting
Higgs sublattice is homothetic to the constant lattice. Finally,
Section~\ref{sec:proof-main} extends the filtration across one-parameter
degenerations and proves the main theorem.

\subsection*{Acknowledgments:} I would like to thank Panagiotis Dimakis for his suggestions and discussions. In our correspondence, Panagiotis shared with me a proof with a similar framework, which he independently obtained in September with the assistance of AI tools. The author takes full responsibility for the accuracy of the above statement.

\subsection*{AI Declaration:} The main idea of this paper originated from discussions between the author and ChatGPT. Precisely, ChatGPT suggested that the author use the theory of the determinant-of-cohomology line bundle to study this conjecture. In addition, ChatGPT was used for calculations and polishing the writing. All
mathematical statements have been independently checked by the author, who takes full
responsibility for the content.

\section{The determinant connection}\label{sec:determinant-connection}

This section constructs the determinant-of-cohomology line bundle for families of flat bundles and equips it with a canonical holomorphic determinant connection. We study its geometric properties, in particular deriving the curvature formula, which will be used in Section~\ref{sec:rees-triangular-exactness} to study the degeneration.

Choose a K\"ahler metric on \(X\), a theta characteristic
\(\kappa^{\otimes2}\simeq K_X\), and on \(\kappa\) the metric whose
square is the induced metric on \(K_X\). Let
\(\mathfrak{Bun}_{n,0}(X)\) be the moduli stack of rank-\(n\), degree-zero
holomorphic vector bundles on \(X\). For the universal bundle
\(\mathcal U\) and the projection
\(\pi:X\times\mathfrak{Bun}_{n,0}(X)\to\mathfrak{Bun}_{n,0}(X)\), define the
\textbf{determinant-of-cohomology line bundle}
\begin{equation}\label{eq:theta-line}
 \Theta:=\det R\pi_*(\mathcal U\otimes\kappa)
 \longrightarrow\mathfrak{Bun}_{n,0}(X).
\end{equation}
For \(V\in\mathfrak{Bun}_{n,0}(X)\), we have
\[
 \Theta_V
 =
 \det R\Gamma(X,V\otimes\kappa)
 =
 \det H^0(X,V\otimes\kappa)
 \otimes
 \det H^1(X,V\otimes\kappa)^{-1}.
\]
Although the dimensions of these cohomology groups may jump with \(V\),
the derived pushforward
\[
 R\pi_*(\mathcal U\otimes\kappa)
\]
is a perfect complex, and hence its determinant $\Theta$ is globally defined as a line bundle by the
Knudsen--Mumford determinant functor
\cite{KnudsenMumford1976}.
Our convention is the dual of that in Bismut--Gillet--Soul\'e:
they define \(\lambda^{\mathrm{KM}}(F):=\det(R\pi_*F)^{-1}\)
\cite[p.~341]{BismutGilletSoule1988III}. Consequently their analytic determinant line, identified with \(\lambda^{\mathrm{KM}}\) by
\cite[Theorem~3.14]{BismutGilletSoule1988III}, is canonically
\(\Theta^\vee\). The Quillen metric and connection on \(\Theta\)
used below are the duals of theirs. 

Let \(B\) be a complex manifold, and let \((\mathcal V,\mathcal D)\) be a family of flat bundles on \(X\) over \(B\).
The underlying bundles determine the classifying map
\[
 q:B\longrightarrow\mathfrak{Bun}_{n,0}(X).
\]
The \textbf{determinant-of-cohomology line bundle} of the family is the pullback
\[
 \lambda_B:=q^*\Theta
 =\det R(\pi_B)_*(\mathcal V\otimes\kappa),
 \qquad \pi_B:=\operatorname{pr}_B:X\times B\to B.\]
For a single bundle or a family over a base, we use the notation
\(\lambda(F):=\det R\Gamma(X,F\otimes\kappa)\) and
\(\lambda_B(\mathcal F):=\det R(\pi_B)_*(\mathcal F\otimes\kappa)\),
respectively.

All analytic parameter spaces needed below are complex manifolds.
Locally on \(B\), trivialize the underlying smooth bundle along \(X\),
holomorphically in the parameter, and choose a Hermitian metric \(h\)
on this fixed smooth bundle, independent of the parameter in \(B\).
For each \(b\in B\), write the corresponding fiber as \((V,D)\),
suppressing \(b\) from the notation, and set
\[
 D_h=\bar\partial_V+\partial_{V,h},\qquad
 D=\bar\partial_V+\partial_D.
\]
and let \(\delta=d_B=\partial_B+\bar\partial_B\) denote exterior differentiation on \(B\). We use the conventions
\[
 \operatorname{ch}(W,\nabla)
 :=\operatorname{tr}\exp\!\left(-\frac{F_\nabla}{2\pi\sqrt{-1}}\right),
 \qquad
 c_1(W,\nabla):=-\frac1{2\pi\sqrt{-1}}\operatorname{tr}F_\nabla.
\]

In expressions involving variations, the parameter directions are
antisymmetrized while the forms on \(X\) are wedged in the displayed
order. For example,
\[
 \begin{split}
 \operatorname{tr}(\delta P\wedge\delta A)(v,w)
 :=\operatorname{tr}\bigl(
 (\delta_vP)\wedge(\delta_wA)
 -(\delta_wP)\wedge(\delta_vA)\bigr).
 \end{split}
\]

The Quillen metric and its Chern connection are constructed in
\cite[Definition~1.5 and Theorem~1.6]{BismutGilletSoule1988III}; we
apply that construction with the dual convention just specified.
Let \(\nabla^{\mathrm Q,h}\) denote the resulting \textbf{Quillen Chern
connection} on \(\lambda_B\), and define the \textbf{determinant connection}
by
\begin{equation}\label{eq:correction}
 \nabla^{\mathrm{det}}_h
 :=\nabla^{\mathrm Q,h}
 +\frac{1}{2\pi\sqrt{-1}}\int_X
   \tr\bigl((\partial_D-\partial_{V,h})
                  \wedge\delta\bar\partial_V\bigr).
\end{equation}
The integral is a scalar one-form on \(B\).

\begin{lemma}[Determinant connection]\label{lem:detconnection}
The connection \(\nabla_h^{\mathrm{det}}\) has the following properties.
\begin{enumerate}
\item In a fixed smooth presentation, \(\nabla_h^{\mathrm{det}}\) is independent
of the chosen parameter-independent Hermitian metric \(h\).

\item The resulting connection is holomorphic, and its curvature is
\begin{equation}\label{eq:correctedcurv}
 F_{\nabla^{\mathrm{det}}}
 =\frac{1}{2\pi\sqrt{-1}}\int_X
   \tr\bigl(\delta\partial_D\wedge\delta\bar\partial_V\bigr).
\end{equation}

\item With the gauge and Lie-derivative conventions used below, its complexified equivariant moment on the ambient space of pairs
\((\bar\partial_V,\partial_D)\) is
\begin{equation}\label{eq:correctedmoment}
 \mu_{\mathrm{det}}(\xi)
 =\frac{1}{2\pi\sqrt{-1}}\int_X\tr(\xi F_D),
 \qquad \xi\in\Omega^0(X,\End V).
\end{equation}
In particular, this moment vanishes for a holomorphic connection on the
curve \(X\).

\item For holomorphic families of connections on \(X\), the resulting
connection is invariant under holomorphic complex gauge transformations
and is functorial under holomorphic base change. Hence the local
constructions glue to a natural holomorphic connection
\(\nabla^{\mathrm{det}}\) on \(\lambda_B\).
\end{enumerate}
\end{lemma}

\begin{proof}
The calculations below are first made for arbitrary pairs
\((\bar\partial_V,\partial_D)\) on the fixed smooth bundle and then
restricted to holomorphic connections, for which \(F_D=0\).

\begin{enumerate}
\item \emph{Independence of the auxiliary metric.}
 Let \(h_u\), \(0\leq u\leq1\), be a path from \(h_0\)
to \(h_1\), independent of \(B\), and put
\(k_u:=h_u^{-1}\partial_u h_u\). For a local holomorphic determinant frame \(\sigma\), let
\(\alpha_{\mathrm Q,h_u}\) denote the Quillen connection form in
the frame \(\sigma\), namely
\[
\alpha_{\mathrm Q,h_u}
=
\partial_B\log\|\sigma\|_{\mathrm Q,h_u}^{\,2}.
\] With our convention, the infinitesimal Bott--Chern anomaly formula
gives the variation of the Quillen metric:
\begin{equation}\label{eq:infinitesimal-anomaly}
 \partial_u\log\|\sigma\|_{\mathrm Q,h_u}^{\,2}
 =
 -\frac1{2\pi\sqrt{-1}}\int_X
 \tr(k_uF_{D_{h_u}});
\end{equation}
see
\cite[Theorem~1.23, Eq.~(1.146), and Remark~1.24]{BismutGilletSoule1988III}. Here the sign incorporates the
dualization of the Bismut--Gillet--Soul\'e determinant convention. Applying \(\partial_B\) to \eqref{eq:infinitesimal-anomaly} yields the corresponding
variation of the Quillen connection form:
\[
 \partial_u\alpha_{\mathrm Q,h_u}
 =
 -\frac1{2\pi\sqrt{-1}}\partial_B
 \int_X\tr(k_uF_{D_{h_u}}).
\]

Since \(\delta k_u=0\), while \(\bar\partial_V\) varies holomorphically
with \(B\),
\[
 \partial_u\partial_{V,h_u}=\partial_{V,h_u}k_u,
 \qquad
 \delta_vF_{D_{h_u}}
 =\partial_{V,h_u}(\delta_v\bar\partial_V)
 \quad(v\in T^{1,0}B).
\]
Integrating by parts on \(X\) gives
\[
 \begin{aligned}
 \partial_u\alpha_{\mathrm Q,h_u}
 &=-\frac1{2\pi\sqrt{-1}}\int_X
   \tr\bigl(k_u\partial_{V,h_u}(\delta\bar\partial_V)\bigr)\\
 &=\frac1{2\pi\sqrt{-1}}\int_X
   \tr\bigl((\partial_u\partial_{V,h_u})\wedge
              \delta\bar\partial_V\bigr).
 \end{aligned}
\]
Integrating in \(u\) yields
\begin{equation*}
 \nabla^{\mathrm Q,h_1}-\nabla^{\mathrm Q,h_0}
 =\frac{1}{2\pi\sqrt{-1}}\int_X
   \tr\bigl((\partial_{V,h_1}-\partial_{V,h_0})
                 \wedge\delta\bar\partial_V\bigr).
\end{equation*}
The correction in \eqref{eq:correction} changes by the negative of this
quantity, proving (i). 

\item \emph{Curvature and holomorphicity.}
The compatible metrics on \(X\) and \(\kappa\) give
\begin{align}\label{eq:ch}
 c_1(\kappa)+\tfrac12c_1(T_X)=0,
\end{align}
so the $(1,1)$ part of
\(\operatorname{ch}(\kappa)\operatorname{td}(T_X)\) vanishes in the
formulas below. Let \(\mathbb F_h\) be the curvature of the total Chern connection
\(D_h+\delta\) on \(X\times B\). With the chosen parameter-independent
smooth trivialization and metric, its relevant components are
\[
 \mathbb F_h=F_{D_h}+\delta\bar\partial_V+\delta\partial_{V,h},
\]
where the last two terms have fiber degree one and base degree one. By
\cite[Theorem~1.27]{BismutGilletSoule1988III}, after dualizing its
determinant convention,
\begin{equation}\label{eq:total-index-curvature}
 \begin{aligned}
 F_{\nabla^{\mathrm Q,h}}
 &=-2\pi\sqrt{-1}\left[\int_X\operatorname{td}(T_X)
                  \operatorname{ch}(\kappa)
                  \tr\exp\!\left(-\frac{\mathbb F_h}{2\pi\sqrt{-1}}\right)
          \right]_{[2]}\\
 &\overset{\eqref{eq:ch}}{=}-\frac1{4\pi\sqrt{-1}}
   \left[\int_X\tr(\mathbb F_h^{\,2})\right]_{[2]},
 \end{aligned}
\end{equation}
where \([2]\) denotes the component of degree \(2\) in the
differential forms on \(B\) obtained after integrating over the
fiber \(X\).

With the chosen parameter-independent smooth trivialization and metric
\(h\), the total connection \(D_h+\delta\) has no base--base curvature
component. Expanding the mixed terms and moving parameter one-forms past fiber
one-forms gives
\[
 \left[\int_X\tr(\mathbb F_h^{\,2})\right]_{[2]}(v,w)
 =-2\int_X\tr\bigl(
  (\delta_v\partial_{V,h})\wedge(\delta_w\bar\partial_V)
 -(\delta_w\partial_{V,h})\wedge(\delta_v\bar\partial_V)\bigr).
\]
Hence \eqref{eq:total-index-curvature} gives
\begin{equation}\label{eq:quillencurv}
 F_{\nabla^{\mathrm Q,h}}
 =\frac1{2\pi\sqrt{-1}}\int_X\tr(\delta\partial_{V,h}\wedge\delta\bar\partial_V).
\end{equation}
Taking the exterior derivative of \eqref{eq:correction} now gives
\eqref{eq:correctedcurv}. For a holomorphic family of pairs,
\(\delta\bar\partial_V\) and \(\delta\partial_D\) have type \((1,0)\)
on \(B\); the correction is therefore of type \((1,0)\) and the
curvature of type \((2,0)\). Thus \(\nabla^{\mathrm{det}}\) is a
holomorphic connection, proving (ii).

\item \emph{The equivariant moment.}
We first record equivariance for parameter-independent complex gauges.
Let \(\mathcal G(V):=\Gamma(X,\operatorname{GL}(V))\) denote the complex gauge group of the fixed
smooth bundle \(V\), and set
\[
 \mathfrak u(V,h):=\{\xi\in\End V:\xi^{*_{h}}=-\xi\},
\]
so that \(\Omega^0(X,\mathfrak u(V,h))\) is the Lie algebra of the
unitary gauge subgroup. Use the left action \(g\cdot D:=gDg^{-1}\), with determinant lift
\(\Phi_g:\lambda_D\to\lambda_{g\cdot D}\). A unitary gauge preserves
the Quillen connection and the correction. If \(s\) is
\(h\)-self-adjoint and \(g_z:=e^{zs}\), the coefficients in the
\(B\)-directions of
\(\Phi_{g_z}^*\nabla^{\mathrm{det}}-\nabla^{\mathrm{det}}\) are
holomorphic in \(z\) and vanish for \(z\in i\mathbb R\); hence they
vanish for all \(z\). Polar decomposition then gives equivariance for
every parameter-independent complex gauge.

The infinitesimal generator of the gauge action is
\[
\xi^\# D
=
\left.\frac{d}{dt}\right|_{t=0}
e^{t\xi}De^{-t\xi}
=
-D^{\mathrm{End}}\xi ,
\]
where \(D^{\mathrm{End}}\xi=[D,\xi]\) denotes the connection induced by \(D\) on
\(\operatorname{End}(V)\). On determinant-line sections set
\[
 (\mathcal L_\xi\sigma)(D)
 :=\left.\frac{d}{dt}\right|_{t=0}
       \Phi_{e^{t\xi}}^{-1}\sigma(e^{t\xi}\cdot D),
 \qquad
 \mu(\xi):=\mathcal L_\xi-\nabla_{\xi^\#}.
\]

For \(\xi\in\Omega^0(X,\mathfrak u(V,h))\), the equivariant
extension of the universal curvature is obtained by replacing
\(\mathbb F_h\) with $\mathbb F_h-\xi.$ Applying the equivariant local index theorem
\cite[Theorem~14 and Remark~19]{Freed2017Equivariant} to the infinitesimal
gauge action generated by \(\xi\), together with the determinant connection formula
\cite[Theorem~1.15]{BismutGilletSoule1988III}, we obtain
\begin{equation}\label{eq:equivariant-index-normalization}
 F_{\nabla^{\mathrm Q,h}}+\mu_{\mathrm Q,h}(\xi)
 =-2\pi\sqrt{-1}\left[\int_X\operatorname{td}(T_X)
                  \operatorname{ch}(\kappa)
                  \tr\exp\!\left(-\frac{\mathbb F_h-\xi}{2\pi\sqrt{-1}}\right)
          \right]_{\mathrm{eq},[2]} .
\end{equation}
Here \([\,\cdot\,]_{\mathrm{eq},[2]}\) denotes total equivariant degree
\(2\), with the Cartan variable \(\xi\) assigned degree \(2\). Combining \eqref{eq:equivariant-index-normalization} and \eqref{eq:total-index-curvature} yields
\begin{equation*}
 \mu_{\mathrm Q,h}(\xi)
 =\frac1{2\pi\sqrt{-1}}\int_X\tr(\xi F_{D_h}),
 \qquad \xi\in\Omega^0(X,\mathfrak u(V,h)).
\end{equation*}

Since \(\xi^\#\bar\partial_V=-\bar\partial_V\xi\), evaluating the
correction on \(\xi^\#\) and integrating by parts gives
\[
 \begin{aligned}
 \mu_{\mathrm{det}}(\xi)
 &=\mu_{\mathrm Q,h}(\xi)
   +\frac1{2\pi\sqrt{-1}}\int_X\tr\!\left(
     \xi\,\bar\partial_V(\partial_D-\partial_{V,h})\right)\\
 &=\frac1{2\pi\sqrt{-1}}\int_X\tr(\xi F_D),
 \end{aligned}
\]
because
\(F_D=F_{D_h}+\bar\partial_V(\partial_D-\partial_{V,h})\) on a curve.
Holomorphicity and the fixed complex-gauge equivariance just proved extend
this identity complex-linearly from \(\mathfrak u(V,h)\) to
\(\Omega^0(X,\End V)\), proving (iii).

\item \emph{Gauge invariance and base change.}
Let \(g:B\to\mathcal G(V)\) be smooth along \(X\) and holomorphic
in the parameter. The variation of \(g\) contributes the vertical
infinitesimal action, hence
\begin{equation*}
 \Phi_g^*\nabla^{\mathrm{det}}_{g\cdot D}
       -\nabla^{\mathrm{det}}_D
 =-\mu_{\mathrm{det}}(g^{-1}\delta g);
\end{equation*}
cf. \cite[Proposition~7.6]{BGV1992}. For a holomorphic family of
connections on \(X\), the right-hand side vanishes by (iii), proving
gauge invariance.

After shrinking \(B\), smooth trivializations holomorphic in the
parameter exist by combining local holomorphic frames with a
parameter-independent partition of unity on \(X\). Any two such
trivializations differ by a gauge transformation \(g:B\to\mathcal G(V)\)
that is smooth along \(X\) and holomorphic in the parameter, so the local
constructions glue. Finally, both the Quillen Chern
connection and the fiber integral in \eqref{eq:correction} commute with
holomorphic base change. This proves (iv).\qedhere
\end{enumerate}
\end{proof}

\begin{remark}\label{rem:tautological}
For a line bundle \(L\) on a complex manifold \(M\), the pullback of
\(L\) to the space of holomorphic connections
\(p:\operatorname{Conn}(L)\to M\) carries a tautological holomorphic
connection whose curvature defines a natural holomorphic symplectic form
on \(\operatorname{Conn}(L)\)
\cite[Proposition~3.3]{BiswasHurtubiseStasheff2012}. In particular, the
difference between two holomorphic connections is a holomorphic one-form.
From this viewpoint, \eqref{eq:correction} amounts to taking the Quillen
connection as a reference connection and adding the corresponding
holomorphic one-form on \(B\).

The result of \cite[Theorem~2.1]{BiswasHurtubise2021} has a related
interpretation on the stable locus. Since their theta line bundle and
normalization conventions differ from ours, we use it only as a conceptual
comparison.
\end{remark}

\section{Simpson filtrations and the extendability of the det frame}\label{sec:rees-triangular-exactness}

In this section, we first recall Simpson filtrations and then use Quillen geometry to translate the behavior of limit points of \(W_{[(E,\theta)]}^1\) into the nondegeneracy of the associated determinant frame. In Section~\ref{sec:determinant-degree}, this nondegeneracy will be converted into a rigidity property of the limiting Higgs structure.

\subsection{The Simpson filtration}\label{subsec:stable-rees}

Simpson characterizes the points of \(W_{[(E,\theta)]}^1\) as flat bundles carrying a Simpson filtration with associated graded Higgs bundle \((E,\theta)\) as below.

\begin{lemma}\label{lem:rees}
Membership in \(W_{[(E,\theta)]}^1\) is equivalent to the existence
of a decreasing filtration by holomorphic subbundles
\begin{equation}\label{eq:filtration}
 V=F^0\supset F^1\supset\cdots\supset F^{s-1}\supset F^s=0,
 \qquad
 D F^p\subset F^{p-1}\otimes K_X
 \quad(1\leq p\leq s-1),
\end{equation}
together with graded identifications
\[
 \phi_i:\Gr_F^{\,s-i}V
 =F^{s-i}/F^{s-i+1}\xrightarrow{\sim}E_i,
 \qquad 1\leq i\leq s,
\]
under which the induced graded Higgs field is \(\theta\).
\end{lemma}

\begin{proof}
This is the stable case of Simpson's Griffiths-transverse/Rees
correspondence. Simpson proves the existence of a gr-semistable
Griffiths-transverse filtration, identifies its associated graded with
the \(\mathbb G_m\)-limit via the Rees construction, and proves
uniqueness up to a common shift of the indices when the limit is stable;
see \cite[Theorem~2.5, Lemma~4.1, Corollary~4.2,
Proposition~4.3]{Simpson2010}. 
\end{proof}

\subsection{The extension-class space}\label{subsec:triangular-presentation}

We introduce the \textbf{finite-dimensional extension-class space} with fixed associated graded \(\bigoplus_i E_i\).

Fix Hermitian metrics on all \(E_i\) and on their orthogonal smooth sum.
Set
\[
 \mathscr E^+:=\bigoplus_{i<j}H^1(X,\Hom(E_j,E_i)),
\]
and represent these cohomology groups by harmonic \((0,1)\)-forms.
Thus \(\mathscr E^+\) is the finite-dimensional complex vector space of
strictly upper-triangular extension parameters. For
\(e=(e_{ij})\in\mathscr E^+\), define
\begin{equation*}
 \bar\partial_{(e)}:=\operatorname{diag}(\bar\partial_{E_1},\ldots,
                    \bar\partial_{E_s})+(e_{ij})_{i<j}.
\end{equation*}
Hence
\(\bar\partial_{(e)}^2=0\) because \(\dim_{\CC}X=1\). Each \(e\in\mathscr E^+\)
defines a holomorphic bundle
\[
E(e):=(E,\bar\partial_{(e)})
\]
equipped with the canonical decreasing filtration
\[
F_{(e)}^p:=\bigoplus_{i=1}^{s-p}E_i,
\qquad 0\leq p\leq s,
\]
for which $\Gr_{F_{(e)}}^{\,s-i}E(e)\simeq E_i.$

\begin{definition}[Constant extensions]\label{def:constant-extensions}
Let \(R\) be a DVR over \(\CC\), and write
\[
 X_R:=\operatorname{Spec}R\times X.
\]
For a vector bundle \(W\) on \(X\), its \textbf{constant extension} to \(X_R\)
is
\[
 W_R:=\operatorname{pr}_X^*W.
\]
We use the same convention for bundle morphisms; in particular,
\[
 (E_R,\theta_R)
 :=
 \left(\bigoplus_{i=1}^s(E_i)_R,\,
       \bigoplus_{i=1}^{s-1}(\theta_i)_R\right).
\]
If \(\Delta^*\) is a punctured disk, write
\[
 X_{\Delta^*}:=\Delta^*\times X,\qquad
 W_{\Delta^*}:=\operatorname{pr}_X^*W,
\]
and similarly denote by \((E_{\Delta^*},\theta_{\Delta^*})\) the
constant extension of \((E,\theta)\) to \(X_{\Delta^*}\).
\end{definition}

\begin{lemma}[Triangular harmonic presentation]\label{lem:presentation}
Let \((\mathcal V,\mathcal D)\) be a holomorphic family of connections on \(X\)
parametrized by a punctured disk \(\Delta^*\), representing points of
\(W_{[(E,\theta)]}^1\). Assume that the Simpson filtrations of
Lemma~\ref{lem:rees} have been chosen holomorphically over \(\Delta^*\):
\[
 \mathcal V=\mathcal F^0\supset\mathcal F^1\supset\cdots
 \supset\mathcal F^{s-1}\supset\mathcal F^s=0,
 \qquad
 \mathcal D\mathcal F^p\subset\mathcal F^{p-1}\otimes K_X
 \quad(1\leq p\leq s-1),
\]
together with holomorphic graded identifications
\[
 \phi_i:\Gr_{\mathcal F}^{\,s-i}\mathcal V
 \xrightarrow{\sim}(E_i)_{\Delta^*},
 \qquad 1\leq i\leq s,
\]
intertwining the induced graded Higgs maps with \(\theta_i\).
Then there exist a single-valued holomorphic map
\[
 e:\Delta^*\longrightarrow\mathscr E^+,\qquad t\longmapsto e(t),
\]
and a smooth isomorphism
\[
 \Psi:\mathcal V\xrightarrow{\sim}E_{\Delta^*},
\]
holomorphic in the parameter, such that for every \(t\in\Delta^*\)
\[
 \Psi\bigl(\mathcal F^p_t\bigr)=F_{(e(t))}^p,
 \qquad
 \operatorname{gr}^{s-i}(\Psi)=\phi_i,
\]
and the Dolbeault operator on \(E\) is \(\bar\partial_{(e(t))}\).
In this presentation,
\begin{equation}\label{eq:lowerblocks}
 (\partial_{\mathcal D})_{i+1,i}=\theta_i,\qquad
 (\partial_{\mathcal D})_{ji}=0\quad(j>i+1).
\end{equation}
\end{lemma}

\begin{proof}
Using the standard Dolbeault description of extensions after choosing a
\(C^\infty\)-splitting, choose
successively a splitting of the filtration that is smooth along \(X\) and
holomorphic in the parameter.  The latter dependence follows by taking
local holomorphic splittings over a Stein cover of \(X\) and gluing them
with a partition of unity independent of \(t\).  In the resulting
identification with \(E_{\Delta^*}\), the Dolbeault operator has fixed
diagonal and strictly upper-triangular entries \(a_{ij}(t)\).

Write \(\bar\partial_0:=\operatorname{diag}(\bar\partial_{E_i})\), and let
\(\mathfrak n^r\) consist of block endomorphisms whose permitted
entries have \(j-i\geq r\). Thus
\begin{equation}\label{eq:nilpotent}
 \mathfrak n^r\mathfrak n^q\subset\mathfrak n^{r+q},
 \qquad \mathfrak n^s=0,
 \qquad
 \bar\partial_0\mathfrak n^r
       \subset\Omega^{0,1}(\mathfrak n^r).
\end{equation}
The current operator is \(\bar\partial:=\bar\partial_0+a\) parametrized by $t$, with
\(a\in\Omega^{0,1}(\mathfrak n^1)\).
For each pair \(i<j\), fix the harmonic projection \(\mathcal H_{ij}\) and a
complex-linear Green operator \(G_{ij}\) satisfying
\[
 \bar\partial_0G_{ij}b=b-\mathcal H_{ij}b.
\]
All \((0,1)\)-forms on a curve are Dolbeault closed, so this
identity applies to every upper entry.

Induct on \(r=1,\ldots,s-1\). Suppose all entries of distance
less than \(r\) are harmonic. Define a smooth endomorphism \(u_r\)
by
\[
 (u_r)_{ij}:=
 \begin{cases}
 -G_{ij}(a_{ij}),&j-i=r,\\
 0,&\text{otherwise},
 \end{cases}
 \qquad g_r:=I+u_r.
\]
The inverse of \(g_r\) is a finite polynomial because \(u_r\)
is nilpotent. The upper part of the transformed operator is
\[
 a^{g_r}:=g_r^{-1}ag_r+g_r^{-1}\bar\partial_0g_r.
\]
By \eqref{eq:nilpotent},
\begin{equation*}
 a^{g_r}-a-\bar\partial_0u_r
          \in\Omega^{0,1}(\mathfrak n^{r+1}).
\end{equation*}
Note that entries of distance less than
\(r\) are unchanged and entries of distance \(r\) become
\(\mathcal H_{ij}a_{ij}\). After \(s-1\) steps every entry is harmonic.
All operations are finite polynomial operations or fixed continuous
complex-linear operations on smooth sections. They preserve
single-valued holomorphic parameter dependence.
For \(s=1\) the induction is empty.

Every gauge is upper unipotent, hence preserves
\(F_{(e)}^\bullet\) and induces the identity on each graded block. One can verify directly that \eqref{eq:lowerblocks} also holds after the above successive gauge transformations.
\end{proof}

\subsection{The connection form}\label{subsec:connection-form}
On the finite-dimensional parameter space \(\mathscr E^+\), we write
\(d\) for the parameter differential previously denoted by \(\delta\).
In this subsection, we prove that the existence of a Simpson filtration (Lemma~\ref{lem:rees}) implies the exactness of the connection one-form of \(\nabla^{\mathrm{det}}\). We mainly use the curvature formula established in Lemma~\ref{lem:detconnection}.

Determinant multiplicativity \cite{KnudsenMumford1976} gives a canonical
identification
\[
 \lambda(E(e))\simeq\bigotimes_{i=1}^{s}\lambda(E_i)
\]
over all \(\mathscr E^+\). Choose a nonzero vector in each fixed factor.
Their product determines an everywhere nonzero holomorphic \textbf{determinant frame}
\(u(e)\), even where cohomology dimensions jump. 

\begin{lemma}\label{lem:extensionform}
There is a holomorphic function
\(f:\mathscr E^+\to\CC\), independent of the choice of connection
on \(E(e)\), such that for a holomorphic connection satisfying
Griffiths transversality for \(F_{(e)}^\bullet\) and whose graded Higgs bundle is
\((E,\theta)\), the connection form of
\(\nabla^{\mathrm{det}}\) in \(u(e)\) is
\begin{equation}\label{eq:exactform}
 df(e)+\frac{1}{2\pi\sqrt{-1}}\sum_{i=1}^{s-1}
             \int_X\tr(\theta_i\wedge de_{i,i+1}).
\end{equation}
In particular this form is exact.
\end{lemma}

\begin{proof}
Write $e=(e_{ij})$ and let \(\alpha_{\mathrm Q}\) be the Quillen connection form in \(u(e)\).
The \((j,i)\)-entry of the Chern connection associated with
\(\bar\partial_{(e)}\), for \(i<j\), is \(-e_{ij}^\dagger\).
Applying \eqref{eq:quillencurv} to the family
\[
e\longmapsto E(e),\qquad e\in\mathscr E^+,
\]
we obtain 
\[
 d\alpha_{\mathrm Q}
       +\frac{1}{2\pi\sqrt{-1}}\sum_{i<j}
          \int_X\tr(d e_{ij}^\dagger\wedge de_{ij})=0.
\] Denote the $(1,0)$-form \[\zeta:=\alpha_{\mathrm Q}
       +\frac{1}{2\pi\sqrt{-1}}\sum_{i<j}
          \int_X\tr( e_{ij}^\dagger\wedge de_{ij}).\]
Hence \(d\zeta=0\) implies \(\zeta\) is holomorphic.
Since \(\mathscr E^+\) is a complex vector space, it follows that
\(\zeta=df(e)\) for a single-valued holomorphic function $f$.

Restrict \eqref{eq:correction} to the parameter space
\(\mathscr E^+\). Since $\delta\bar\partial_{(e)}=(de_{ij})_{i<j}$, the trace in the correction term is
\[
\tr\!\left(
(\partial_D-\partial_{(e),h})\wedge d\bar\partial_{(e)}
\right)
=
\sum_{i<j}
\tr\!\left(
(\partial_D-\partial_{(e),h})_{ji}\wedge de_{ij}
\right).
\]
By \eqref{eq:lowerblocks} and
\((\partial_{(e),h})_{ji}=-e_{ij}^{\dagger}\), we have
\[
(\partial_D-\partial_{(e),h})_{ji}
=
\begin{cases}
\theta_i+e_{i,i+1}^{\dagger},& j=i+1,\\[1mm]
e_{ij}^{\dagger},& j>i+1.
\end{cases}
\]
Hence \eqref{eq:correction}, together with $df(e)=\zeta$ gives \eqref{eq:exactform}. 
\end{proof}

\begin{remark}
For \(s=1\), the space \(\mathscr E^+\) is a point, the sums are empty,
and the determinant form is zero. All other cases use the same exactness
calculation.
\end{remark}

\subsection{Extendability of the determinant frame}\label{subsec:determinant-frame}
We apply the exactness from Section~\ref{subsec:connection-form} to a punctured one-parameter family. It forces the determinant frame induced by the filtration and the graded identifications to extend holomorphically and without zeros across the puncture.
\begin{proposition}\label{prop:valuation}
Let \(\Delta\) be a disk with coordinate \(t\), and let
\((\mathcal V,\mathcal D)\) be a holomorphic family of connections on \(X\)
parametrized by \(\Delta\). Suppose that over
\(\Delta^*:=\Delta\setminus\{0\}\) a Simpson filtration and graded
identifications have been chosen so that
Lemma~\ref{lem:presentation} applies. Let
\[
e:\Delta^*\longrightarrow\mathscr E^+
\]
be the resulting holomorphic map. If the determinant frame
\(u(e(t))\) is meromorphic at \(0\), then it extends across \(0\)
to a nowhere-vanishing holomorphic frame.
\end{proposition}

\begin{proof}
By Lemma~\ref{lem:presentation}, after identifying
\(\mathcal V|_{X\times\Delta^*}\) with \(E_{\Delta^*}\), the Dolbeault
operator is \(\bar\partial_{(e(t))}\).
By the gauge invariance of Lemma~\ref{lem:detconnection}, we may compute
the determinant connection in this presentation.

Choose a nowhere-zero holomorphic frame \(u_0\) of the determinant line
on \(\Delta\). Since \(u(e(t))\) is meromorphic at \(0\), after
shrinking \(\Delta\) we may write
\[
u(e(t))=t^w a(t)u_0,
\]
where \(w\in\mathbb Z\) and \(a\) is holomorphic and nowhere zero.
If \(\alpha_0\) is the connection form in \(u_0\), then the connection
form in \(u(e(t))\) is
\[
 \alpha_0+w\frac{dt}{t}+d\log a(t).
\]
Since \(\alpha_0\) is holomorphic on \(\Delta\), its period around \(0\)
is therefore \(2\pi\sqrt{-1}\,w\).

On the other hand, Lemma~\ref{lem:extensionform} gives this connection
form as
\[
d\left(
f(e(t))+\frac{1}{2\pi\sqrt{-1}}\sum_i
\int_X\tr(\theta_i\wedge e_{i,i+1}(t))
\right).
\]
The function in parentheses is single-valued and holomorphic on
\(\Delta^*\), so this period is zero. Hence \(w=0\), and \(u(e(t))\)
extends holomorphically and nonvanishingly across \(0\).
\end{proof}

\section{Determinant degree and Higgs sublattice}\label{sec:determinant-degree}
In this section, we turn the extendability result of Proposition~\ref{prop:valuation} into a homothety-rigidity statement for the limiting Higgs lattice (defined below): it must lie in the same homothety class as the constant family of Higgs bundles; see Lemma~\ref{lem:scalarlattice}. This rigidity will be used directly in the next section to prove the main theorem.

Let \(R\) be a DVR with uniformizer \(t\) and fraction field \(K\),
and use the constant extension \((E_R,\theta_R)\) of
Definition~\ref{def:constant-extensions}. A \textbf{coherent Higgs sublattice}
\(\HH\subset E_R\) means a coherent \(\OO_{X_R}\)-submodule such that
\[
 \theta_R(\HH)\subset\HH\otimes K_X,\qquad
 \HH[1/t]:=\HH\otimes_RK=E_R\otimes_RK=:E_R[1/t].
\]
No local-freeness assumption on \(\HH\) is made in this definition.
We suppress pullbacks of bundles from \(X\) to \(X_R\) from the notation.

\begin{lemma}[Determinant--degree identity]\label{lem:detdegree}
Let \(\Delta\) be a disk with coordinate \(t\), let
\(R:=\mathcal O_{\Delta,0}\), and let \(\HH\subset E_R\) be a coherent
Higgs sublattice. Put
\[
 Q:=E_R/\HH,\qquad Q_j:=t^jQ/t^{j+1}Q.
\]
Suppose that the meromorphic determinant frame \(u(e(t))\) on
\(\Delta^*\) is identified, via determinant multiplicativity and
\(\HH[1/t]=E_R[1/t]\), with the fixed determinant frame of \(E_R\).
Then
\[
 \sum_j\deg Q_j=-\ord_0\bigl(u(e(t))\bigr).
\]
\end{lemma}

\begin{proof}
Set
\[
 C_Q:=R\Gamma(X_R,Q\otimes\kappa).
\]
Since \(Q\) is supported on the special fiber, \(C_Q\) is a torsion
perfect complex over \(R\). From
\(0\to\HH\to E_R\to Q\to0\) and our convention
\(\lambda:=\det R\Gamma\), determinant multiplicativity gives
\[
 \ord_0\bigl(u(e(t))\bigr)
 =-\sum_q(-1)^q\operatorname{length}_R H^q(C_Q).
\]
Using the \(t\)-adic filtration of \(Q\) and additivity of Euler
characteristic,
\[
 \ord_0\bigl(u(e(t))\bigr)
 =-\sum_j\chi(X,Q_j\otimes\kappa)
 =-\sum_j\deg Q_j,
\]
where the last equality follows from Riemann--Roch and
\(\deg\kappa=g-1\).
\end{proof}

The following DVR argument is in the spirit of Langton's valuative method
\cite{Langton1975}; the scalar conclusion below uses in addition the
stability of the fixed Higgs bundle.

\begin{lemma}\label{lem:scalarlattice}
Let \(R\) be a DVR over \(\CC\) with uniformizer
\(t\) and residue field \(\CC\). Suppose that
\(\HH\subset E_R\) is a coherent Higgs sublattice, and that
\(Q:=E_R/\HH\) is killed by a power of \(t\). Put
\[
 Q_j:=t^jQ/t^{j+1}Q.
\]
If \(\sum_j\deg Q_j=0\), then there exists an integer \(m\geq0\)
such that $\HH=t^mE_R.$
\end{lemma}

\begin{proof}
For \(\bar v\in E=E_R/tE_R\), choose a lift \(v\in E_R\) and set
\[
 \psi_j(\bar v):=[t^jv]\in Q_j=t^jQ/t^{j+1}Q.
\]
Changing the lift by \(tw\) changes \(t^jv\) by \(t^{j+1}w\), so
\(\psi_j\) is well defined. It is surjective, and it is a Higgs map
because \(\HH\) is a Higgs sublattice. Define \(K_j:=\ker(\psi_j)\). Thus we have $Q_j\cong E/K_j$. Multiplication by \(t\) gives \(K_j\subset K_{j+1}\).

Each \(Q_j\) is a coherent Higgs quotient of the stable degree-zero
Higgs bundle \((E,\theta)\). We \textbf{claim} that $\deg Q_j\geq 0,$ with equality only for \(Q_j=0\) or \(Q_j\simeq E\).

\noindent\textbf{Case 1: \(0<\rk K_j<n\).}
Stability gives $\deg K_j\leq\deg K_j^{\mathrm{sat}}<0,$
hence
\[
\deg Q_j=-\deg K_j>0.
\]

\noindent\textbf{Case 2: \(\rk K_j=n\).}
If \(K_j\neq E\), then \(Q_j\) is a nonzero torsion sheaf, so
\[
\deg Q_j=\operatorname{length}(Q_j)>0.
\]
If \(K_j=E\), then \(Q_j=0\).

\noindent\textbf{Case 3: \(\rk K_j=0\).}
Then \(K_j=0\), hence \(Q_j\simeq E\) and \(\deg Q_j=0\).

This proves the claim. Since we assume $\sum\deg Q_j=0$, every \(Q_j\) has degree zero.
Thus, for some \(m\),
\[
 K_j=0\ (j<m),\qquad K_j=E\ (j\geq m).
\]
The equality \(Q_m=0\) gives \(t^mQ=t^{m+1}Q\). Iterating and
using nilpotence yields \(t^mQ=0\), hence \(t^mE_R\subset\HH\).
Conversely, if \(\HH\subset t^jE_R\) and \(j<m\), write its
local section as \(t^jw\). The reduction of \(w\) belongs to
\(K_j=0\), so the section is divisible by \(t^{j+1}\).
Induction proves \(\HH\subset t^mE_R\).
\end{proof}

In the setting of Proposition~\ref{prop:specialization},
Proposition~\ref{prop:valuation} gives
\(\ord_0 u(e(t))=0\). Lemma~\ref{lem:detdegree} then yields
\(\sum_j\deg Q_j=0\), so Lemma~\ref{lem:scalarlattice} applies.

\section{Proof of Theorem~\ref{thm:main}}\label{sec:proof-main}

We first prove Simpson's closedness conjecture for a one-parameter degeneration.

\begin{proposition}\label{prop:specialization}
Let \((C,c_0)\) be a smooth pointed algebraic curve, let \(\mathcal V\)
be an algebraic vector bundle on \(X\times C\), and let
\(\mathcal D:\mathcal V\to\mathcal V\otimes\operatorname{pr}_X^*K_X\) be an algebraic relative
connection along \(X\), defined across \(c_0\). Suppose, after a
possible finite base change, its generic connection is equipped with a Griffiths-transverse filtration whose graded Higgs bundle is
identified with \((E,\theta)\). Then its special connection
lies in \(W_{[(E,\theta)]}^1\).
\end{proposition}

\begin{proof}
Work over the local DVR \(R\), with fraction field \(K\) and uniformizer
\(t\). Write the generic Simpson filtration in decreasing form
\[
 \mathcal V_K=F_K^0\supset F_K^1\supset\cdots\supset F_K^s=0,
 \qquad
 \Gr_F^{\,s-i}\mathcal V_K\simeq E_{i,K},
\]
and intersect every step with the regular family:
\[
 A^p:=\mathcal V_R\cap F_K^p,\qquad
 G_i:=A^{s-i}/A^{s-i+1}.
\]
Thus \(A^0=\mathcal V_R\) and \(A^s=0\). Let
\(j:X_K\hookrightarrow X_R\) be the inclusion of the generic fiber. Then
\[
 A^p=\ker\!\left(\mathcal V_R\longrightarrow
 j_*(\mathcal V_K/F_K^p)\right).
\]
Since \(j\) is quasi-compact and quasi-separated,
\(j_*(\mathcal V_K/F_K^p)\) is quasi-coherent; hence \(A^p\) is a
quasi-coherent subsheaf of the coherent sheaf \(\mathcal V_R\), and therefore
coherent because \(X_R\) is Noetherian
\cite[Tags~01LC and~01Y1]{StacksProject}.

Since \(\mathcal V_R/A^p\) embeds into \(j_*(\mathcal V_K/F_K^p)\) and
\(\mathcal V_K/F_K^p\) is torsion-free, the quotient \(\mathcal V_R/A^p\) is
torsion-free.  At every closed point of the regular surface \(X_R\),
the depth lemma applied to
\[
0\longrightarrow A^p\longrightarrow\mathcal V_R
\longrightarrow\mathcal V_R/A^p\longrightarrow0
\]
gives \(\operatorname{depth}A^p\geq2\)
\cite[Tag~0BUV]{StacksProject}.  Since \(A^p\subset\mathcal V_R\) is
torsion-free, it follows that \(A^p\) satisfies Serre's condition
\((S_2)\), and hence is reflexive
\cite[Tag~0AY6]{StacksProject}.  As \(X_R\) is regular of dimension
two, every coherent reflexive sheaf on \(X_R\) is locally free
\cite[Tag~0B3N]{StacksProject}.  Thus each \(A^p\) is locally free. Each \(G_i\) is torsion-free, because it embeds in
\(\mathcal V_R/A^{s-i+1}\), but is not yet assumed locally free.

For \(1\leq p\leq s-1\), the obstruction to
\[
\mathcal D A^p\subset A^{p-1}\otimes K_X
\]
lies in $H^0\!\left(
X_R,\,
\mathcal Hom\!\left(A^p,\mathcal V_R/A^{p-1}\right)
\otimes \operatorname{pr}_X^*K_X
\right).$ It vanishes after inverting \(t\),
and its target has no \(t\)-torsion, so it vanishes.
The connection therefore induces \(\OO\)-linear graded Higgs maps, for \(i<s\),
\[
 G_i\longrightarrow G_{i+1}\otimes K_X.
\]
The last graded Higgs map is zero. 
The generic identifications \(G_{i,K}\simeq E_{i,K}\) intertwine
these graded Higgs maps with \(\theta_i\). Viewed as sections of the coherent
homomorphism sheaves on \(X_R\), their possible poles are supported on
the Cartier divisor \(t=0\). Hence, by coherence and quasi-compactness,
for each \(i\) some power of \(t\) clears all poles; taking the maximum
of these exponents gives one common power \(t^a\). The resulting maps
are injections
\[
 G_i\hookrightarrow E_{i,R}
\]
that still intertwine the graded Higgs maps.  Consequently
\[
 \HH:=\bigoplus_iG_i\hookrightarrow E_R
\]
is a coherent Higgs sublattice. Its quotient \(Q:=E_R/\HH\) is killed by
a power of \(t\).

Every coherent sheaf in use is perfect on the regular surface,
and its proper derived pushforward to \(R\) is perfect.
Determinant multiplicativity \cite{KnudsenMumford1976} along the
filtration gives
\[
 \lambda(\mathcal V_R)\simeq\bigotimes_i\lambda(G_i)=\lambda(\HH).
\]
Thus the determinant frame induced by the generic filtration and graded
identifications is the frame to which Lemma~\ref{lem:detdegree} applies.
Choose a small analytic disk \(\Delta\) around \(c_0\), with coordinate
\(t\), and set \(\Delta^*:=\Delta\setminus\{0\}\). After shrinking,
the algebraic generic filtration and graded identifications spread over
\(\Delta^*\). The family \((\mathcal V,\mathcal D)\) is holomorphic over \(\Delta\), and
the induced determinant frame \(u(e(t))\) is meromorphic at
\(0\). By Proposition~\ref{prop:valuation} and
Lemma~\ref{lem:detdegree}, we have
\[
\sum_j\deg Q_j=0.
\]
Hence Lemma~\ref{lem:scalarlattice} gives $\HH=t^mE_R.$ Because the embedding is block diagonal, each
\(G_i\) maps isomorphically to \(t^mE_{i,R}\). Divide every
identification by this same scalar \(t^m\); the resulting
isomorphisms \(G_i\simeq E_{i,R}\) still intertwine all the graded Higgs maps.
In particular every \(G_i\) is locally free. Successive extensions
then show that every \(\mathcal V_R/A^p\) is locally free. We now have a
genuine subbundle filtration, including on the special fiber, with
graded Higgs bundle \((E,\theta)\). Lemma~\ref{lem:rees} gives the
asserted special scaling limit.
\end{proof}
We first record an irreducibility observation that will be used later in the proof
of the main theorem.

Any flat bundle \((V,D)\) admitting the Simpson filtration
\(F^\bullet\) of Lemma~\ref{lem:rees} is \textbf{irreducible}. Indeed, a nonzero
proper invariant subbundle \(H\) has degree zero, and intersecting it with
\(F^\bullet\) gives a proper Higgs subsheaf
\[
 \Gr_{H\cap F}H\hookrightarrow\Gr_F V\simeq E,
 \qquad \deg\Gr_{H\cap F}H=\deg H=0.
\]
Its saturation is a nonzero proper Higgs subbundle of nonnegative degree, contradicting stability of \((E,\theta)\). Thus every point of \(W_{[(E,\theta)]}^1\) is represented by an irreducible
flat bundle. 
\begin{proof}[Proof of Theorem~\ref{thm:main}]
By \cite[Proposition~5.1]{Simpson2010},
\(W_{[(E,\theta)]}^1\) is locally closed, hence constructible. We use Simpson's algebraic framed parameter scheme with its universal family
and good quotient by a reductive group
\cite[Theorem~4.7 and Section~4]{Simpson1994I}. Thus we have
\[
 q:R^{\mathrm{ss}}\longrightarrow\Mdr(X,n),
 \qquad
 S:=q^{-1}(W_{[(E,\theta)]}^1),
\]
where \(R^{\mathrm{ss}}\) is the semistable framed parameter space in Simpson's
construction. By definition, \(S\) is invariant under the
reductive group action. Let
\[
 Z:=\overline{S}.
\]
Then \(Z\) is a closed invariant
subset.

Let \(y\) be a closed point of $\overline{W_{[(E,\theta)]}^1}$ and choose a closed
\(z\in Z\) above it. If \(z\in S\), there is nothing to prove.
Otherwise choose an irreducible component \(Z_0\) through \(z\).
Since \(S\) is open dense in \(Z\), \(S\cap Z_0\) is open dense in
\(Z_0\). Taking an irreducible component of general hyperplane
sections through \(z\) in an affine neighborhood gives an integral
curve through \(z\) not contained in \(Z_0\setminus S\).
Normalizing this curve and choosing a point over \(z\) gives a
smooth pointed algebraic curve \((C,c_0)\to Z\) whose generic point
lies in \(S\). Pull back the universal family to \(X\times C\). This is an actual algebraic family, not
merely a curve of coarse moduli points.

Write \(K:=\CC(C)\). Over \(\overline K\), Lemma~\ref{lem:rees} therefore supplies the actual
generic connection with a Griffiths-transverse filtration and graded
identification with \((E,\theta)_{\overline K}\).
By finite-presentation descent, the filtration inclusions and graded
isomorphisms, together with their inverses, descend from
\(\overline K\) to some finite extension \(K'/K\)
\cite[Tag~01ZR]{StacksProject}. The quotient
sheaves remain locally free by faithfully flat descent; exactness,
Griffiths transversality, and compatibility with \(\theta\) can all
be checked after extension to \(\overline K\). Normalize \(C\)
in \(K'\) and choose a point above \(c_0\). The pulled-back family remains defined there, while
the generic filtered data spread over a punctured neighborhood.
Proposition~\ref{prop:specialization} now places its special connection
in \(W_{[(E,\theta)]}^1\), so \(y\in W_{[(E,\theta)]}^1\).

Finally, a nonempty constructible subset of a finite-type complex
scheme contains a closed point
\cite[Tags~02J6 and~005X]{StacksProject}. Applied to
\(\overline{W_{[(E,\theta)]}^1}\setminus W_{[(E,\theta)]}^1\), this
shows that the boundary is empty, proving the theorem.
\end{proof}

\end{document}